\documentclass[a4paper,12pt,intlimits,oneside]{amsart}
\usepackage[utf8]{inputenc}
\usepackage{amsmath}
\usepackage{amsthm}
\usepackage{latexsym}
\usepackage{amssymb}
\usepackage{xcolor}
\usepackage{dsfont}
\usepackage{mathrsfs}
\usepackage[colorlinks=true]{hyperref}
\numberwithin{equation}{section}
\numberwithin{figure}{section}
\def\expo_#1{{\rm e}^{#1}}

\def\R{{\mathbb R}}
\def\C{{\mathbb C}}

\def\1{1\!{\rm l}}

\def\build#1_#2^#3{\mathrel{\mathop{\kern 0pt#1}\limits_{#2}^{#3}}}
\def\td_#1,#2{\mathrel{\mathop{\build\longrightarrow_{#1\rightarrow #2}^{}}}}

\newcommand{\ben}{\begin{equation}}
\newcommand{\een}{\end{equation}}
\newcommand{\beno}{\begin{eqnarray*}}
\newcommand{\eeno}{\end{eqnarray*}}

\newtheorem{theorem}{Theorem}
\newtheorem{corollary}{Corollary}
\newtheorem{proposition}{Proposition}
\newtheorem{lemma}{Lemma}
\newtheorem{remark}{Remark}

\date{}
\title{UG}
\newcounter{rea}
\newcounter{reb}
\newcounter{res}
\title[Korenblum and Bloch spaces ]{Some equivalent definitions for weighted Korenblum  and Bloch-type spaces of the upper half-plane}
\author{Benoit F. Sehba \address{Benoit F. Sehba, Department of Mathematics, University of Ghana, PO. Box LG 62 Legon, Accra, Ghana }}

\begin{document}
\maketitle
\begin{abstract}
  In this paper, we prove some equivalent characterizations of weighted Korenblum spaces and Bloch spaces in terms of symbols of bounded Hankel operators.
\end{abstract}
\section{Introduction}
\medskip
In the upper half-plane $\C_+:=\{z\in\C: \Im mz>0\}$, for $\Omega$ a positive function, and $t\in\mathbb{R}$, we denote by $L_{\Omega,t}^\infty(\C_+)$, the space of all measurable functions $f$ such that
\begin{equation}\label{eq:Linfty}
\Vert f\Vert_{L_{\Omega,t}^\infty}:=\sup_{z\in\C_+}\,\left(\Im mz\right)^t\Omega(z)\vert f(z)\vert<\infty.
\end{equation}
If $\Omega(z)=1$ for all $z\in \C_+$, then the above space is just denoted $L_{t}^\infty(\C_+)$, and the corresponding norm $\Vert f\Vert_{L_t^\infty}$. 

We denote by $\mathcal H(\C_+)$ the set of holomorphic functions in $\C_+$. For $t\geq 0$, the (weighted) Korenblum space  $H_{\Omega,t}^\infty(\C_+)$ is the intersection $\mathcal H(\C_+)\cap L_{\Omega,t}^\infty(\C_+)$. 

We note that $H_0^\infty(\C_+)=H^\infty(\C_+)$ is the usual space of bounded holomorphic functions on $\C_+$. For $t>0$, the space $H_t^\infty(\C_+)$ is not trivial. Indeed, the function $f(z)=\frac{1}{(z+i)^t}$ belongs to $H_t^\infty(\C_+)$. Also, if $H^\infty(\C_+)$ contains constants, the only constant in $H_t^\infty(\C_+)$ for $t\neq 0$ is $f(z)=0$.
\vskip .1cm

Let us introduce some related classes, starting with the Bergman spaces.
Let $\alpha>-1$, and let $1\le p<\infty$. For $\Omega$ a weight on $\C_+$, we denote by $L_{\Omega,\alpha}^{p}(\C_+)$ the set of all measurable functions $f$ on $\mathbb{C}_+$ such that
$$\Vert f\Vert_{\Omega,\alpha,p}:=\left(\int_{\C_+}|f(x+iy)|^p\Omega(x+iy) y^\alpha dxdy\right)^{1/p}<\infty.$$
\medskip

 For $1\le p<\infty$ and $\alpha>-1$,
the Bergman space $A_{\Omega,\alpha}^{p}(\mathbb{C}_+)$ is the intersection $L_{\Omega,\alpha}^{p}(\mathbb{C}_+)\cap \mathcal H(\C_+)$. 
\medskip 

For $\alpha=0$, the above spaces become $L_{\Omega}^{p}(\mathbb{C}_+)$ and $A_{\Omega}^{p}(\mathbb{C}_+)$. The corresponding norm is denoted $\Vert \cdot \Vert_{\Omega,p}$.
\medskip

When $\Omega(z)=1$ for all $z\in\C_+$, the above spaces are denoted $A_{\alpha}^{p}(\mathbb{C}_+)$, and the corresponding norm is denoted $\Vert \cdot\Vert_{\alpha,p}$.
The usual Bergman spaces correspond to the case $\alpha=0$ and are just denoted $A^p(\C_+)$ and the corresponding norm $\Vert \cdot \Vert_{p}=\Vert \cdot\Vert_{0,p}$.
\medskip

It is well known (see, for example, \cite{BBGNPR}) that there exists a constant $C>0$ such that for any $f\in A_{\alpha}^{p}(\mathbb{C}_+)$ and any $z\in \C_+$,
\begin{equation}\label{eq:pointwiseberg}
\vert f(z)\vert\leq C\left(\Im mz\right)^{-\frac{2+\alpha}{p}}\Vert f\Vert_{\alpha,p}.
\end{equation}
This shows that the Bergman space $A_{\alpha}^{p}(\mathbb{C}_+)$ is continuously embedded in the Korenblum space $H_{\frac{2+\alpha}{p}}^\infty(\C_+)$.
\medskip

Another related space is the Bloch space $\mathcal{B}(\C_+)$ which is the space of all holomorphic functions $b$ on $\C_+$ such $b'\in H_{1}^\infty(\C_+)$. We endow $\mathcal{B}(\C_+)$ with the norm
$$\|b\|_{\mathcal{B}}:=\vert b(i)\vert+\Vert b'\Vert_{L_{1}^\infty}$$
and note that it is then a Banach space of functions and not equivalence classes. 
\medskip

The space $\mathcal{B}(\C_+)$ obviously contains constants. For our own convenience, we put
$$\|f\|_{\dot{\mathcal{B}}}:=\sup_{z\in \C_+}(\Im mz)|f'(z)|.$$
We know from \cite{BGS1} that there exists a constant $C>0$ such that for any Bloch function $f$, it holds that
$$\vert f(z)\vert\leq C\omega(z)\|f\|_{\mathcal{B}}$$
where $\omega(z)=\ln_+(\vert z\vert)+\ln_+(1/ \Im mz)$. Recall that for $s>0$, $\ln_+s=\max\{0,\ln s\}$.
\medskip

Let $k\in\mathbb{R}$. The $\omega^k$-Bloch space $\mathcal{B}_{\omega^k}(\C_+)$ is the space of all holomorphic functions $b$ on $\C_+$ such that $b'\in H_{\omega^k,1}^\infty(\C_+)$. We endow $\mathcal{B}_{\omega^k}(\C_+)$ with the norm
$$\|b\|_{\mathcal{B}_{\omega^k}}:=\vert b(i)\vert+\Vert b'\Vert_{L_{\omega^k,1}^\infty}$$
and note that it is then a Banach space of functions. It also contains constants, and we put
$$\|f\|_{\dot{\mathcal{B}}_{\omega^k}}:=\sup_{z\in \C_+}(\Im mz)\omega^k(z)|f'(z)|.$$

\medskip

In this note, for $\Omega=\omega^k$ for some $k\in\R$, we give equivalent definitions of the above weighted Korenblum spaces and the weighted Bloch spaces in terms of sets of symbols of bounded Hankel operators. Other equivalent definitions are in terms of dual of weighted Bergman spaces (for weighted Bloch spaces, see \cite{BGS,SH}).
\medskip

This work is a kind of spin-off of \cite{BGS} with its own perspective. It is motivated by the challenges encountered in obtaining the necessary and sufficient conditions for the symbols of Hankel operators in \cite{BGS}. 
\section{Statement of the results}
For simplicity, we will be using the notation $$dV_\alpha(z)=\left(\Im m(z)\right)^\alpha dV(z)=y^\alpha dxdy\quad \textrm{for}\quad z=x+iy.$$ 
\medskip

For $\alpha>-1$, the Bergman projection $P_\alpha$ is the orthogonal projection from $L_\alpha^2(\mathbb{C}_+)$ into its closed subspace $A_\alpha^2(\mathbb{C}_+)$. This operator is  given by $$P_\alpha f(z)=\int_{\mathbb{C}_+}K_\alpha(z,w)f(w) dV_\alpha(w)$$
where $$K_\alpha(z,w)=\frac{c_\alpha}{\left(z-\bar{w}\right)^{2+\alpha}}$$
is the Bergman kernel, and $c_\alpha$ a constant that depends only on $\alpha$. We use the notation $P=P_0$.
\medskip

It is not difficult to prove that for $\alpha>-1$ and $\beta>-1$, the Bergman projection $P_\beta$ is bounded on $L_\alpha^p(\C_+)$, $1\leq p<\infty$, if and only if
\begin{equation}\label{eq:boundednessbergcondit}
1+\alpha<p(1+\beta)
\end{equation}
(see, for example, \cite{BanSeh,BBGNPR,Sehba}). In this case, one can show that $P_\beta$ also reproduces elements of $A_\alpha^p(\C_+)$. Another consequence of this fact is that for $1\leq p<\infty$, the differentiation operator $\frac{d^k}{dz^k}$ is bounded from $A_\alpha^p(\C_+)$ to $A_{\alpha+kp}^p(\C_+)$ with the equivalence of the norms.
\medskip

 For $b$ a bounded function in $\C_+$, the Hankel operator $h_b^\alpha$ is the operator defined on $A_\alpha^2(\C_+)$ by $$h_b^{\alpha}(f)=P_{\alpha}(b\overline{f}), f\in A_\alpha^2(\C_+).$$ 
This operator is clearly well defined on $A_\alpha^2(\C_+)$. It is even bounded on $A_\alpha^2(\C_+)$, at least as long as $b$ is bounded. This means, among others, that for any $g\in A_\alpha^2(\C_+)$, we have the control
$$\vert\langle h_b^\alpha(f),g\rangle_\alpha\vert\leq C_b\Vert f\Vert_{\alpha,2}\Vert g\Vert_{\alpha,2}.$$
Here $\langle \cdot,\cdot\rangle_\alpha$ stands for the Hermitian inner product of $L_\alpha^2(\C_+)$. For $\alpha=0$, we write $h_b=h_b^0$ and $\langle\cdot,\cdot\rangle$ for the inner product of $L^2(\C_+)$.
\vskip .2cm
We want to relate the Hankel operator to the Korenblum space associated with $A_\alpha^2(\C_+)$ with $\alpha> -1$, that is, $H_t^\infty(\C_+)$ for $t$ large enough. 

We make the following observation: still for $f\in A_\alpha^2(\C_+)$ and $g\in A_\alpha^2(\C_+)$, we first recall that
$$\langle h_b^\alpha(f),g\rangle_\alpha = \langle P_\alpha(b\overline{f}),g\rangle_\alpha=\langle b,fg\rangle_\alpha.$$
This can be rewritten as 
\begin{equation}\label{eq:hankel1}\langle h_b^\alpha(f),g\rangle_\alpha =\langle (\Im m\cdot)^{\alpha-t}b,fg\rangle_t\end{equation}
where now, $\langle \cdot,\cdot\rangle_t$ stands for the Hermitian inner product of $L_t^2(\C_+)$. That is,
\begin{equation}\label{eq:innerprod1}\langle F,G\rangle_t:=\int_{\C_+} F(z)\overline{G(z)}(\Im mz)^tdV(z).\end{equation}
Returning to (\ref{eq:hankel1}), we have that $a(z)=(\Im mz)^{\alpha-t}b(z)$ belongs to $L_{t-\alpha}^\infty(\C_+)$ since $b$ was taken bounded. Using the fact that $P_t$ reproduces functions in $A_\alpha^2(\C_+)$, it follows that
\begin{equation}\label{eq:hankel11}\langle h_b^\alpha(f),g\rangle_\alpha =\langle (P_t(a\overline{f}),g\rangle_t=\langle h_{a}^t(f),g\rangle_t.\end{equation}
\vskip .1cm
We have that $P_t(a\overline{f})$ makes sense. Indeed, as $f\in A_\alpha^2(\C_+)$ and $a\in L_{t-\alpha}^\infty(\C_+)$, the function $a\overline{f}$ belongs to $\in L_{2t-\alpha}^2(\C_+)$ and $P_t$ is well defined and bounded in the last space for $t$ large enough.
\vskip .1cm
We see that since $f$ and $g$ are taken in $A_\alpha^2(\C_+)$, the term in the last equality in (\ref{eq:hankel11}) forces $h_{a}^t(f)$ to be in the predual of $A_\alpha^2(\C_+)$ with respect to the inner product (\ref{eq:innerprod1}). 
\medskip

It is known that the Hankel operator $h_b^\alpha$ is bounded on $A_\alpha^p(\C_+)$ for $1<p<\infty$, if and only if $b=P_\alpha g$ for some $g\in L^\infty(\C_+)$ (see, for example, \cite{NanaSehba}). To extend this characterization to weighted Bloch spaces in consideration, we first prove the following characterization of Korenblum spaces.
\begin{theorem}\label{thm:main2}
Let $1< p<\infty$, and let $t\geq 1$. Assume that $l,k\in\mathbb{R}$. Let $b$ be a holomorphic function on $\C_+$.  Then the following assertions are equivalent.
\begin{enumerate}
    \item[(a)] $b\in H_{\omega^{l+k},t}^\infty(\C_+)$.
    \item[(b)] The Hankel operator  $h_b^t$ is bounded from  $A_{\omega^{-lp}}^p(\C_+)$ to $A_{\omega^{kp},tp}^p(\C_+)$. Moreover,
    $$\Vert h_b^t\Vert_{A_{\omega^{-lp}}^p\to A_{\omega^{kp},tp}^p}\simeq \Vert b\Vert_{H_{\omega^{l+k},t}^\infty}.$$
\end{enumerate}
\end{theorem}
From the above result, we deduce the following.
\begin{corollary}\label{cor:main2}
Let $1< p<\infty$, $l,k\in\mathbb{R}$. Let $b$ be a holomorphic function on $\C_+$ such that $b=Pb$.  Then the following assertions are equivalent.
\begin{enumerate}
    \item[(a)] $b\in \mathcal{B}_{\omega^{l+k}}(\C_+)$.
    \item[(b)] The Hankel operator $h_b$ is bounded from  $A_{\omega^{-lp}}^p(\C_+)$ to $A_{\omega^{kp}}^p(\C_+)$. Moreover,
    $$\Vert b\Vert_{\dot{\mathcal B}_{\omega^{l+k}}}\lesssim \Vert h_b\Vert_{A_{\omega^{-lp}}^p\to A_{\omega^{kp}}^p}\lesssim \Vert b\Vert_{\mathcal{B}_{\omega^{l+k}}}.$$
\end{enumerate}
\end{corollary}
The corollary is obtained from the theorem by making the following observation. For $b$ as in the corollary (say for $l=k=0$ and $t=1$), by the above theorem, $h_{b'}^1$ maps $A^p(\C_+)$ to $A_p^p(\C_+)$. We note that $A_p^p(\C_+)$ is the space of derivatives of functions in $A^p(\C_+)$. We then recover $h_bf$ in this case by noting that it is the solution of the differential equation. 
\begin{equation}\label{eq:diffhankel0}
\frac{d}{dz}\left(h_bf(z)\right)=c\left(h_{b'}^1f\right)(z).
\end{equation}
We can speak of the solution of (\ref{eq:diffhankel0}) since the constants do not belong to $A^p(\C_+)$ for any $1\leq p<\infty$. 
This approach for characterizing symbols of bounded Hankel operators on Bergman spaces seems new.
\medskip

The following is a kind of end-point version of Theorem \ref{thm:main2}.
\begin{theorem}\label{thm:main4}
Let $k\leq 0$, and let $l\in \R$. Assume that $0< \varepsilon<t$ for $t\geq 1$. Let $b$ be a holomorphic function on $\C_+$.  Then the following assertions are equivalent.
\begin{enumerate}
    \item[(a)] $b\in H_{\omega^{l+k},t-\varepsilon}^\infty(\C_+)$.
    \item[(b)] The Hankel operator  $h_b^t$ is bounded from  $A_{\omega^{l}}^1(\C_+)$ to $A_{\omega^{-k},t-\varepsilon}^1(\C_+)$. Moreover,
    $$\Vert h_b^t\Vert_{A_{\omega^{l}}^1\to A_{\omega^{-k},t-\varepsilon}^1}\simeq \Vert b\Vert_{H_{\omega^{l+k},t-\varepsilon}^\infty}.$$
\end{enumerate}
\end{theorem}
If we take $t=1$ and let $\varepsilon\to 0$ in the last theorem, we are led to the same type of problems as those encountered in \cite{BGS}.
Recall that for $b\in H^\infty(\C_+)$, that $h_bf\in A^2(\C_+)$ for any $f\in A^2(\C_+)$, is equivalent to saying that for any $g\in A^2(\C_+)$, 
\begin{equation}\label{eq:hankA2}\langle h_b(f),g\rangle=\langle b,fg\rangle
\end{equation}
and its modulus is bounded above by a constant time $\Vert f\Vert_2\Vert g\Vert_2$ .
\medskip

 For $f\in A^1(\C_+)$, can we understand $h_bf$ as in (\ref{eq:hankA2}) with $g\in \mathcal{B}(\C_+)$? In the arxiv.org version of \cite{BGS}, $h_bf$ for $f\in A^1(\C_+)$, was defined as the holomorphic function such that for any $g\in \mathcal{B}(\C_+)$, 
$$\langle h_bf,g\rangle:=\langle b,fg\rangle.$$
To make this definition unconditional, that is, such that the right-hand side above is the same independently of the choice of the representative $g$ of a class in  Bloch space, one sees that $\langle b,f\rangle=0$ should hold. This forced the authors of \cite{BGS} (see the arxiv.org version of the paper) to consider studying $$h_b: A_b^1(\C_+)\to A^1(\C_+)$$ where $\displaystyle A_b^1(\C_+)=\{f\in A^1(\C_+):\langle b,f\rangle=0\}$. Unfortunately, this does not give a necessary and sufficient condition on the symbol of the operator.
\medskip

In Section 3, we provide some useful results we need in our presentation. Some dual characterizations of the weighted Bergman spaces are given in Section 4. We prove our results in Section 5. 
\medskip

Throughout the text, we use the notation $A\lesssim B$ (respectively, $A\gtrsim B$) whenever there exists a uniform constant $C>0$ such that $A\le CB$ (respectively, $A\ge CB$). The notation $A\simeq B$ means that $A\lesssim B$ and $A\gtrsim B$.
\section{Useful results}
\medskip

We refer to \cite{BGS} for the following.
\begin{lemma}\label{lem:weightforellirudin}
Let $\alpha>0$, $\beta>-1$, and $k\in \R$. Let $$\Omega(z)=\left(\ln^{\varepsilon_1}(e+|z|)+\ln^{\varepsilon_2}\left(e+\frac{1}{\Im mz}\right)\right)^{ k},$$
$\varepsilon_j\in\{0,1\}$, $j=1,2$. Then there is a constant $C=C_{\alpha,\beta}>0$ such that for any $z_0\in \mathbb{C}_+$,
\begin{equation}\label{eq:weightforellirudin1}
    \int_{\mathbb{C}_+}\frac{\Omega(z)dV_\beta(z)}{|z-\overline z_0|^{2+\alpha+\beta}}\le C(\Im m(z_0))^{-\alpha}\Omega(z_0).
\end{equation}
\end{lemma}
The first part of the following follows from Fubini's theorem and Lemma \ref{lem:weightforellirudin}. Its second part can be proved as in \cite[Proposition 1.7]{BGS}. 
\begin{proposition}\label{prop:boudednessP1}
Let $\alpha> 0$, and $\beta>-1$, let $k\in\mathbb{R}$. The Bergman projector $P_{\alpha+\beta}$ is bounded from $L_{\omega^{-k},\beta}^1(\mathbb{C}_+)$ to $L_{\omega^{-k},\beta}^1(\mathbb{C}_+)$. If, moreover, $\alpha+\beta>0$, then $P_{\alpha+\beta}$ reproduces the functions in $A_{\omega^{-k},\beta}^1(\mathbb{C}_+)$
\end{proposition}
Let us deduce the following.
\begin{corollary}\label{cor:boundednessP1omegainverse}
Let $k\in \mathbb{R}$, $\alpha>0$ and $\beta>-1$. Assume that $\alpha+\beta>0$. The Bergman projection $P_\alpha$, reproduces the functions in $A^1_{\omega^{k},\beta}(\mathbb{C}_+)$.
\end{corollary}
\begin{proof}
As $\alpha+\beta>0$, we have from Proposition \ref{prop:boudednessP1} that $P_{\alpha+\beta}$ reproduces the functions in $A^1_{\omega^{k},\beta}(\mathbb{C}_+)$. The conclusion then follows from the fact that the reproducing kernel of $A_{\alpha+\beta}^2(\mathbb{C}_+)$ belongs to $A^1(\mathbb{C}_+)$ and that, as $\alpha>0$, $P_\alpha$ reproduces the functions in $A^1(\mathbb{C}_+)$.
\end{proof}
The following is from \cite{DOS,DST}.
\begin{proposition}\label{prop:boudednessPbeta}
Let $\alpha> -1$, and, $\beta>-1$, let $k\in\mathbb{R}$. Assume $1\leq p<\infty$. Then, the Bergman projector $P_{\beta}$ is bounded on $L_{\omega^{-k},\alpha}^p(\mathbb{C}_+)$ if and only if $1+\alpha<p(1+\beta)$. Moreover, $P_\beta$ reproduces the functions in $A_{\omega^{-k},\alpha}^p(\mathbb{C}_+)$.
\end{proposition}

Combining the above facts, one obtains the following.
\begin{corollary}\label{cor:hardyineq}
Let $k\in \mathbb{R}$, $\beta>-1$. Assume $1\leq p<\infty$. Then $f\in A^p_{\omega^{k},\beta}(\mathbb{C}_+)$ if and only if $f'\in A^p_{\omega^{k},\beta+p}(\mathbb{C}_+)$ with equivalence of the norms.
\end{corollary}
The following follows from Lemma \ref{lem:weightforellirudin}.
\begin{proposition}\label{prop:boudednessP1infty}
Let $\alpha\geq t> 0$ and $k\in \R$. The Bergman projector $P_{\alpha}$ is then bounded from $L_{\omega^{k},t}^\infty(\mathbb{C}_+)$ to $L_{\omega^{k},t}^\infty(\mathbb{C}_+)$.
\end{proposition}
The following pointwise estimate is needed (see \cite[Lemma 1.4]{BGS}).
\begin{lemma}\label{lem:pointwisebloch}
Let $k\in \R$. Then there is a constant $C_k>0$ such that for any $f\in \mathcal{B}_{\omega^{k}}(\mathbb{C}_+)$ of norm $1$,
$$
|f(z)| \leq\left\{
    \begin{array}{lll}
      C_k & \mbox{if } k>1, \\
        C_{-1} \ln(\omega(z))  & \mbox{if } k=1,\\
        C_k (\omega(z))^{1-k } & \mbox{if } k<1.
    \end{array}
\right.$$

\end{lemma}
The following test functions has proved to be useful in \cite{BGS}; we need them here too.
\begin{lemma} \label{lem: testTheta}
 We consider the function 
 $$(z, w)\mapsto \theta_w(z):= 1-\log(z-\bar{w})+\ln|i+w|+ 2\log(i+z).$$
 It has the following properties.
 \begin{itemize}
 \item [(i)]
 $\Re e (\theta_w)>1-\ln 2+\ln (|i+z|),$ so that, in particular, its logarithm and its powers are well defined.
 \item [(ii)]
 $$|\theta_w(z)|\simeq 1+\ln_+(|z|) +\ln_+\frac 1{|z-\overline w|}$$
 with constants that do not depend on $w\in\mathbb C_+$.
 \item [(iii)]
   For all $k\leq 0$ and $\varepsilon\geq 0$,  the function $z\mapsto \frac{\theta_w^{-k}(z)}{(z-\overline{w})^\varepsilon}$ belongs to $H_{\omega^{k},\varepsilon}^\infty(\C_+)$ and $$\left\Vert\frac{\theta_w^{-k}}{(\cdot-\overline{w})^\varepsilon}\right\Vert_{H_{\omega^{k},\varepsilon}}\simeq 1$$ 
 with constants that do not depend of $w$.
  \end{itemize}
 \end{lemma}
 \begin{proof}
 Assertions (i) and (ii) are from \cite[Lemma 1.5]{BGS1}. The proof of assertion (iii) uses (ii), and in fact it is proved in the lines of \cite[Lemma 1.5]{BGS1} (see page 14 of this reference).
 \end{proof}
\section{Dual of weighted Bergman spaces}
The following is elementary, as follows from the usual arguments.
\begin{proposition}\label{prop:dualA22}
Let $1<p<\infty$, $t\geq 1$, and $k\in\mathbb{R}$. The dual space $\left( A_{\omega^{-kp},tp}^p(\C_+)\right)^*$ of the Bergman space $A_{\omega^{-kp},tp}^p(\C_+)$ identifies with $A_{\omega^{kp'}}^{p'}(\C_+)$ (pp'=p+p') under the duality pairing (\ref{eq:innerprod1}).
\end{proposition}

Let us prove the following.
\begin{proposition}\label{prop:dualA1}
Let $k\in \R$, and $0\leq \varepsilon< t$, $t\geq 1$. Then the  dual space $\left( A_{\omega^{-k},\varepsilon}^1(\C_+)\right)^*$ of $A_{\omega^{-k},\varepsilon}^1(\C_+)$
identifies with $H_{\omega^k,t-\varepsilon}^\infty(\C_+)$ under the duality pairing (\ref{eq:innerprod1}).
\end{proposition}
\begin{proof}
That every $b\in H_{\omega^k,t-\varepsilon}^\infty(\C_+)$ defines an element of $\left( A_{\omega^{-k},\varepsilon}^1(\C_+)\right)^*$ through (\ref{eq:innerprod1}) is obvious. Now suppose that $\Lambda\in \left( A_{\omega^{-k},\varepsilon}^1(\C_+)\right)^*$. Then we can extend $\Lambda$ to a linear operator on $L_{\omega^{-k},\varepsilon}^1(\C_+)$ that we still denote $\Lambda$. Hence there is a $g\in L^\infty(\C_+)$ such that for any $f\in L_{\omega^{-k},\varepsilon}^1(\C_+)$,
$$\Lambda(f):=\langle f,g\rangle:=\int_{\C_+}f(z)\overline{g(z)}(\Im mz)^\varepsilon\omega^{-k}(z) dV(z).$$
In particular, for any $f\in A_{\omega^{-k},\varepsilon}^1(\C_+)$,
$$\Lambda(f):=\langle f,g\rangle=\int_{\C_+}f(z)\overline{g(z)}(\Im mz)^{-t+\varepsilon}\omega^{-k}(z)dV_t(z).$$
As $P_t$ reproduces elements of $A_{\omega^{-k},\varepsilon}^1(\C_+)$, it easily follows that
\begin{eqnarray*}\Lambda(f) &:=& \langle f,(\Im m\cdot)^{-t+\varepsilon}\omega^{-k} g\rangle_t\\ &=& \langle P_t(f),(\Im m\cdot)^{-t+\varepsilon}\omega^{-k}g\rangle_t\\ &=& \langle f,P_t\left((\Im m\cdot)^{-t+\varepsilon}\omega^{-k}g\right)\rangle_t.
\end{eqnarray*}
Let us put $b=P_t\left((\Im m\cdot)^{-t+\varepsilon}\omega^{-k}g\right)$. We note that as $g\in L^\infty(\C_+)$, $(\Im m\cdot)^{-t+\varepsilon}\omega^{-k}g\in L_{\omega^{k},t-\varepsilon}^\infty(\C_+)$. Thus, by Proposition \ref{prop:boudednessP1infty}, $b\in H_{\omega^{k},t-\varepsilon}^\infty(\C_+)$. This concludes the proof.
\end{proof}
\section{Boundedness of Hankel operators}
\subsection{Atomic decomposition of $A_{\omega^{-k}}^1(\C_+)$ and weak  factorization}

We give here a generalization of the atomic decomposition of the Bergman space $A_{\omega^{-1}}^1(\C_+)$ obtained in \cite{BGS1}. We then obtain weak factorizations of the Bergman spaces $A_{\omega^{-k}}^1(\C_+)$ that will be used in our characterizations of symbols of bounded Hankel operators.
\medskip

The following is obtained as in \cite[Proposition 11]{BGS1}.
\begin{proposition}\label{prop:atomicdecompo}
Let $k\in \mathbb{R}$. Let $f$ be a holomorphic function in $A^1_{\omega^{-k}}(\C_+)$. For any $\alpha>0$, there exists a sequence of complex numbers $\{c_k\}$ and a sequence of points $\{w_j\}$ in $\C_+$ such that
\begin{equation}\label{eq:atomicdecompo}f(z)=\sum_{j=0}^\infty \frac{c_j(\Im m (w_j))^{\alpha}\omega^k ( w_j)}{(z-\overline w_j)^{2+\alpha}}
\end{equation}
with
$\sum |c_j|\simeq \|f\|_{L^1(\omega^{-k})}.$
\end{proposition}

For $0\leq\varepsilon<t$, we consider the following function.
\begin{equation}\label{eq:atom1}
    f_\varepsilon(z)=\frac{(\Im m (w))^{t-\varepsilon}\omega^{-(l+k)} (w)}{(z-\overline w)^{2+t}}.
\end{equation} 
We have the following factorization.
\begin{proposition}\label{prop:factor0}
Let $l,k\in \mathbb{R}$. Let $f_0$ be an atom given by (\ref{eq:atom1}). Then there exist $g\in A_{\omega^{-lp}}^p(\C_+)$ and $\theta \in A_{\omega^{-kp}}^{p'}(\mathbb{C}_+)$, such that $$f_0=g\times \theta$$
with 
$$\Vert g\Vert_{A_{\omega^{-lp}}^p}\times \Vert \theta\Vert_{A_{\omega^{-kp}}^{p'}}\lesssim 1.$$
\end{proposition}
\begin{proof}
Consider the following functions
$$g(z)=\frac{\left(\Im m w\right)^{\frac tp}\omega^l(w)}{\left(z-\overline{w}\right)^{\frac {2+t}p}}\quad\text{and}\quad \theta(z)=\frac{\left(\Im m w\right)^{\frac t{p'}}\omega^k(w)}{\left(z-\overline{w}\right)^{\frac {2+t}{p'}}}.$$
Then clearly, $f_0=g\times\theta$ and
$\Vert g\Vert_{A_{\omega^{-lp}}^p}\lesssim 1$ and $\Vert \theta\Vert_{A_{\omega^{-kp'}}^{p'}}\lesssim 1$.
\end{proof}
Combining Proposition \ref{prop:atomicdecompo} and Proposition \ref{prop:factor0}, we obtain the following.
\begin{proposition}\label{prop:weakfacto0}
Let $l,k\in \mathbb{R}$. The following assertions are satisfied.
\begin{itemize}
\item[(a)] The product of a function $f\in A_{\omega^{-lp}}^p(\C_+)$ and  a function $g \in A_{\omega^{-kp'}}^{p'}(\mathbb{C}_+)$, belongs $A_{\omega^{-(l+k)}}^1(\C_+)$ and 
$$\Vert fg\Vert_{A_{\omega^{-(l+k)}}^1}\leq \Vert f\Vert_{A_{\omega^{-lp}}^p}\Vert g\Vert_{A_{\omega^{-kp'}}^{p'}}.$$
\item[(b)] Any function $f\in A_{\omega^{-(l+k)}}^1(\C_+)$ can be written as 
\begin{equation}\label{eq:weakfacto0}
f=\sum_{j=0}^\infty f_jg_j,\quad f_j\in A_{\omega^{-lp}}^p(\C_+)\quad\text{and}\quad g_j\in A_{\omega^{-kp'}}^{p'}(\mathbb{C}_+).
\end{equation}
Moreover,
\begin{equation}\label{eq:weakfacto00}
\sum_{j=0}^\infty\Vert f_j\Vert_{A_{\omega^{-lp}}^p}\Vert g_j\Vert_{A_{\omega^{-kp'}}^{p'}}\lesssim \Vert f\Vert_{A_{\omega^{-(l+k)}}^1}.
\end{equation}
\end{itemize}
\end{proposition}
\begin{proof}
(a) is harmless. For assertion (b), using (\ref{eq:atomicdecompo}), it suffices to take 
\begin{equation}\label{eq:testfacto0}f_j(z)=c_j^{1/p}\frac{\left(\Im m w_j\right)^{\frac tp}\omega^l(w_j)}{\left(z-\overline{w_j}\right)^{\frac {2+t}p}}\quad\text{and}\quad g_j(z)=c_j^{1/p'}\frac{\left(\Im m w_j\right)^{\frac t{p'}}\omega^k(w_j)}{\left(z-\overline{w_j}\right)^{\frac {2+t}{p'}}}.
\end{equation}

\end{proof}
The atom in (\ref{eq:atom1}) can also be factorized as follows.
\begin{proposition}\label{prop:factor1}
Let $k\leq 0$, $0<\varepsilon< t$. Let $f:=f_\varepsilon$ be an atom given by (\ref{eq:atom1}). Then there exist $g\in A_{\omega^l}^1(\C_+)$ and $\theta \in H_{\omega^{k},\varepsilon}^{\infty}(\mathbb{C}_+)$, such that $$f=g\times \theta$$
with 
$$\Vert g\Vert_{A_{\omega^l}^1}\times \Vert \theta\Vert_{H_{\omega^k,\varepsilon}^\infty}\lesssim 1.$$
\end{proposition}
\begin{proof}
We take $\theta(z)=\frac{\theta_w^{-k}(z)}{(z-\overline{w})^\varepsilon}$, where $\theta_w$ is the function given in Lemma \ref{lem: testTheta}. The function $g$ is then given as follows $$g(z)=g_w(z)=\frac{(\Im m (w))^{t-\varepsilon}\omega^{-(l+k)} (w)}{(z-\overline w)^{2+t-\varepsilon}\theta_w^{-k}(z)}.$$ We know from Lemma \ref{lem: testTheta} that $\theta\in H_{\omega^k,\varepsilon}^{\infty}(\mathbb{C}_+)$ with $\Vert \theta\Vert_{H_{\omega^k,\varepsilon}^{\infty}}\lesssim 1$. That $\Vert g\Vert_{A_{\omega^l}^1}\lesssim 1$ is obtained as in \cite[Proposition 1.12]{BGS}.
\end{proof}
Combining Proposition \ref{prop:atomicdecompo} and Proposition \ref{prop:factor1}, we obtain the following.
\begin{proposition}\label{prop:weakfacto1}
Let $l\mathbb{R}$, $k\leq 0$, and $0<\varepsilon<t$. The following assertions are satisfied.
\begin{itemize}
\item[(a)] The product of a function $f\in A_{\omega^{l}}^1(\C_+)$ and  a function $g \in H_{\omega^{k},\varepsilon}^{\infty}(\mathbb{C}_+)$, belongs to such that $A_{\omega^{(l+k)},\varepsilon}^1(\C_+)$ and 
$$\Vert fg\Vert_{A_{\omega^{(l+k)},\varepsilon}^1}\leq \Vert f\Vert_{A_{\omega^{l}}^1}\Vert g\Vert_{H_{\omega^{k},\varepsilon}^{\infty}}.$$
\item[(b)] Any function $f\in A_{\omega^{(l+k)}}^1(\C_+)$ can be written as 
\begin{equation}\label{eq:weakfacto1}
f=\sum_{j=0}^\infty f_jg_j,\quad f_j\in A_{\omega^{l}}^p(\C_+)\quad\text{and}\quad g_j\in H_{\omega^{k},\varepsilon}^{\infty}(\mathbb{C}_+).
\end{equation}
Moreover,
\begin{equation}\label{eq:weakfacto11}
\sum_{j=0}^\infty\Vert f_j\Vert_{A_{\omega^{l}}^1}\Vert g_j\Vert_{H_{\omega^{k},\varepsilon}^{\infty}}\lesssim \Vert f\Vert_{A_{\omega^{(l+k)}}^1}.
\end{equation}
\end{itemize}
\end{proposition}

\medskip

\subsection{Proof of the main results}
Let us start with the proof of Theorem \ref{thm:main2}.
\medskip

\begin{proof}[Proof of Theorem \ref{thm:main2}]
We first assume that $b\in H_{\omega^{l+k},t}^\infty(\C_+)$. Recall that the dual space of $A_{\omega^{kp},tp}^p(\C_+)$ is $A_{\omega^{-kp'}}^{p'}(\C_+)$ under the duality pairing (\ref{eq:innerprod1}). We have that the product of a function $f\in A_{\omega^{-lp}}^p(\C_+)$ and a function $g\in A_{\omega^{-kp'}}^{p'}(\C_+)$ belongs to $A_{\omega^{-(l+k)}}^1(\C_+)$. We also have that the dual space of $A_{\omega^{-(l+k)}}^1(\C_+)$ under the pairing (\ref{eq:innerprod1}) is $H_{\omega^{l+k},t}^\infty(\C_+)$. As $P_t$ reproduces functions in $A_{\omega^{-kp'}}^{p'}(\C_+)$, we obtain that
\begin{eqnarray*} \vert\langle h_b^t(f),g\rangle_t\vert=\vert\langle b,fg\rangle_t\vert &\leq& \Vert b\Vert_{H_{\omega^{l+k},t}^\infty}\Vert fg\Vert_{A_{\omega^{-(l+k)}}^1}\\ &\leq& \Vert b\Vert_{H_{\omega^{l+k},t}^\infty}\Vert f\Vert_{A_{\omega^{-lp}}^p}\Vert g\Vert_{A_{\omega^{-kp'}}^{p'}}.
\end{eqnarray*}
Thus for $b\in H_{\omega^{l+k},t}^\infty(\C_+)$, $h_b^t$ is bounded from $A_{\omega^{-lp}}^p(\C_+)$ to $A_{\omega^{kp},tp}^p(\C_+)$. Moreover,
$\Vert h_b^t\Vert_{A_{\omega^{-lp}}^p\longrightarrow A_{\omega^{kp},tp}^p}\leq \Vert b\Vert_{H_{\omega^{l+k},t}^\infty}.$
\medskip

For the converse statement, we assume that $h_b^t$ is well defined and bounded from $A_{\omega^{-lp}}^p(\C_+)$ to $A_{\omega^{kp},tp}^p(\C_+)$ and we prove that necessarily, $b\in H_{\omega^{l+k},t}^\infty(\C_+).$  By Proposition \ref{prop:dualA1}, it is sufficient to prove that, for $F$ in a dense subset of $A^1_{\omega^{-(l+k)}}(\C_+)$, we have \begin{equation} \label{forlog}
    |\langle b, F\rangle _t|=\left|\int_{\C_+}b(z)\overline{F(z)}(\Im m z)^t dV(z)\right|\leq C\|F\|_{A^1_{\omega^{-(l+k)}}}
\end{equation} for some uniform constant $C>0$.  
We consider the dense subset of functions in $A^1_{\omega^{-(l+k)}}(\C_+)$ with finite atomic decomposition,
$$F=\sum_{finite} f_jg_j,$$ 
where $f_j$ and $g_j$ are given in (\ref{eq:testfacto0})
with $\displaystyle\sum_{finite} \Vert f_j\Vert_{A_{\omega^{-lp}}^p}\Vert g_j\Vert_{A_{\omega^{-kp'}}^{p'}}\leq C \|F\|_{A^1_{\omega^{-(l+k)}}}.$
We use the above to conclude that $$\langle b, F\rangle _t=\sum_{finite} \langle b, f_jg_j\rangle _{t}.$$
The boundedness of $h_b^t$ from $A_{\omega^{-lp}}^p(\C_+)$ to $A_{\omega^{kp},tp}^p(\C_+)$ implies that
$$|\langle b, f_jg_j\rangle_t|=|\langle h_b^t( f_j), g_j\rangle_t|\leq C $$
for some uniform constant $C.$ Inequality \eqref{forlog} then follows from this fact.
The proof is complete.
\end{proof}
\begin{remark}\label{rem:remark1}
\medskip

\begin{itemize}
\item[(a)] We have that $$h_b^t=h_{P_tb}^t\quad\text{on}\quad A_{-lp}^p(\C_+).$$ Indeed, for $f\in A_{\omega^{-lp}}^p(\C_+)$, we have
\begin{eqnarray*}
h_{P_tb}^tf(z) &=&\int_{\C_+}P_tb(w)\overline{\left(\frac{f(w)}{(\overline{z}-w)^{2+t}}\right)}dV_t(w)\\ &=& \int_{\C_+}b(w)\overline{P_t\left(\frac{f(w)}{(\overline{z}-w)^{2+t}}\right)}dV_t(w).
\end{eqnarray*}
Now, put $g(w)=\frac{f(w)}{(\overline{z}-w)^{2+t}}$. As $f\in A_{\omega^{-lp}}^p(\C_+)$, $g\in A_{\omega^{-lp},(2+t)p}^p(\C_+)$. As $P_t$ reproduces the functions in $A_{\omega^{-lp},(2+t)p}^p(\C_+)$, we conclude that $P_tg=g$. Hence $h_{P_tb}^tf=h_b^tf$.
\item[(b)] From the above observation and Theorem \ref{thm:main2}, we can conclude that $$b\in H_{\omega^{l+k},t}^\infty(\C_+)\quad\text{if and only if}\quad P_tb\in H_{\omega^{l+k},t}^\infty(\C_+).$$
\end{itemize}
\end{remark}
Let us prove the following kind of integration by parts.
\begin{lemma}\label{lem:integparts}
Let $k,l\in\R$. Let $b$ be a holomorphic function such that $b=Pb$, and let $1\leq p<\infty$. Then there is a constant $c\in \mathbb{R}$ such that for all  $f\in A^1(\C_+)\cap A_{\omega^l}^p(\C_+)$,
\begin{equation}\label{eq:integparts}
\int_{\C_+}f(z)\overline{b(z)}dV(z)=
c\int_{\C_+}f(z)\overline{b'(z)}\left(\Im mz \right)dV(z)
\end{equation}
whenever the left hand-side makes sense.
\end{lemma}
\begin{proof}
To obtain (\ref{eq:integparts}), we use the following ingredients:
\begin{itemize}
\item $\int_{\C_+}f(z)dV(z)=0$ for $f\in A^1(\C_+)$;
\item $P_1f=f$ for any $f\in A_{\omega^l}^p(\C_+)$;
\item $P_2(b')=c_0b'$ so that $b=c_1P_1((\Im m\cdot)b')+c_2$ for some $c_j\in \C$, $j=0,1,2$.
\end{itemize}
It follows from the above and Fubini's theorem that
\begin{eqnarray*}
\int_{\C_+}f(z)\overline{b(z)}dV(z) &=& c_1\int_{\C_+}f(z)\overline{P_1((\Im m\cdot)b')(z)}dV(z)\\ &=& c_1\int_{\C_+}\left(\int_{\C_+}\frac{f(z)}{\left(w-\bar{z}\right)^3}dV(z)\right)\overline{b'(w)}\left(\Im m w\right)^2dV(w)\\ &=& C\int_{\C_+}\left(Pf\right)'(w)\overline{b'(w)}\left(\Im m w\right)^2dV(w)\\ &=& C\int_{\C_+}\left(P_1f\right)'(w)\overline{b'(w)}\left(\Im m w\right)^2dV(w)\\ &=& C\int_{\C_+}f(z)\overline{P_2b'(z)}\left(\Im m z\right)dV(z)\\ &=& C\int_{\C_+}f(z)\overline{b'(z)}\left(\Im mz \right)dV(z).
\end{eqnarray*}
\end{proof}
Let us prove Corollary \ref{cor:main2}.
\begin{proof}[Proof of Corollary \ref{cor:main2}]
If we assume that $b\in \mathcal{B}_{\omega^{l+k}}(\C_+)$ with $b=Pb$, then  using Lemma  \ref{lem:integparts}, one obtains that for some constant $c$,
\begin{equation}\label{eq:diffhankel}
\frac{d}{dz}\left(h_{b}f(z)\right)=c\left(h_{b'}^1f\right)(z)\quad\text{for all}\quad f\in  A_{\omega^{-lp}}^p(\C_+).
\end{equation}
Indeed, we have \begin{eqnarray*}\overline{\frac{d}{dz}\left(h_{b}f(z)\right)} &=& 
c\int_{\C_+}\frac{f(w)}{\left(\overline{z}-w\right)^3}\overline{b(w)}dV(w).
\end{eqnarray*}
It suffices then to see using H\"older's inequality and Lemma \ref{lem:weightforellirudin}, that the function $w\mapsto \frac{f(w)}{\left(\overline{z}-w\right)^3}$ is in $A^1(\C_+)\cap A_{\omega^s}^q(\C_+)$ for some $q>1$ and $s\in\R$. Also, using the pointwise estimates in Lemma \ref{lem:pointwisebloch}, H\"older's inequality and Lemma \ref{lem:weightforellirudin}, one obtains that for each $z\in\C_+$ fixed, $P_1\left(b\overline{f}\right)(z)$ makes sense.
\medskip

It follows using the previous proposition that for $b$ holomorphic,
\begin{eqnarray*}
b\in \mathcal{B}_{\omega^{l+k}}(\C_+) &\Longleftrightarrow& b'\in H_{\omega^{l+k},1}^\infty(\C_+)\\ &\Longleftrightarrow& h_{b'}^1f\in A_{\omega^{kp},p}^p(\C_+)\quad\text{for all}\quad f\in A_{\omega^{-lp}}^p(\C_+)\\ &\Longleftrightarrow& \left(h_{b}f\right)'\in A_{\omega^{kp},p}^p(\C_+)\quad\text{for all}\quad f\in A_{\omega^{-lp}}^p(\C_+)\\ &\Longleftrightarrow& h_{b}f\in A_{\omega^{kp}}^p(\C_+)\quad\text{for all}\quad f\in A_{\omega^{-lp}}^p(\C_+).
\end{eqnarray*}
The proof is complete.
\end{proof}
Let us prove Theorem \ref{thm:main4}.
\begin{proof}[Proof of Theorem \ref{thm:main4}]
We first assume that $b\in H_{\omega^{-(l+k)},t-\varepsilon}^\infty(\C_+)$.  Then, we have that for any $f\in A_{\omega^l}^1(\C_+)$, the product $b\overline{f}$ belongs to $L_{\omega^{-k},t-\varepsilon}^1(\C_+)$ on which $P_t$ is defined and bounded.
\medskip

Let us show that $h_b^t(f)$ belongs to $A_{\omega^{-k},t-\varepsilon}^1(\C_+)$ for $f\in A_{\omega^l}^1(\C_+)$. By Proposition \ref{prop:dualA1}, it is enough to test it against functions $g\in H_{\omega^k,\varepsilon}^\infty(\C_+)$. We get again
$$\langle h_b^t(f),g\rangle_t=\langle P_t(b\overline{f}),g\rangle_t=\langle b\overline{f},P_tg\rangle_t=\langle b\overline{f},g\rangle_t=\langle b,fg\rangle_t.$$
We observe that for $f\in A_{\omega^l}^1(\C_+)$ and $g\in H_{\omega^k,\varepsilon}^\infty(\C_+)$, the product $fg$ belongs to $A_{\omega^{l+k},\varepsilon}^1(\C_+)$. It follows that
\begin{eqnarray*}
\vert \langle h_b^t(f),g\rangle_t\vert &\leq& \Vert b\Vert_{L_{\omega^{-(l+k)},t-\varepsilon}^\infty}\Vert fg\Vert_{L_{\omega^{l+k},\varepsilon}^1}\\ &\leq& \Vert b\Vert_{H_{\omega^{-(l+k)},t-\varepsilon}^\infty}\Vert f\Vert_{L_{\omega^l}^1}\Vert g\Vert_{L_{\omega^k,\varepsilon}^\infty}.
\end{eqnarray*}
Conversely, suppose that $b$ is such that $h_b^t(f)$ is well defined for any $f\in A_{\omega^l}^1(\C_+)$ and $h_b^t$ is bounded from this space to $A_{\omega^{-k},t-\varepsilon}^1(\C_+)$. Then proceeding as for Theorem \ref{thm:main2}, using Proposition \ref{prop:weakfacto1}, one obtains that
$b\in H_{\omega^{-(l+k)},t-\varepsilon}^\infty(\C_+)$. 
The proof is complete.
\end{proof}
\section{conclusion}
In this paper, we have obtained a characterization in terms of symbols of bounded Hankel operators between weighted Bergman spaces, of functions in weighted Korenblum spaces and weighted Bloch spaces that are reproduced by the Bergman projection. Our aim was not to study Hankel operators on their own and hence we focused only on cases that provide equivalent definition of Korenblum and Bloch spaces.
\medskip

The author would like to thank Aline Bonami and Sandrine Grellier for fruitfull discussions on the topics of this paper.
\bibliographystyle{plain}

\end{document}